\documentclass[11pt, twoside, leqno]{article}
\usepackage{amssymb}
\usepackage{amsmath}
\usepackage{amsthm}
\usepackage{color}
\usepackage{mathrsfs}
\usepackage{graphicx}
\usepackage{indentfirst}
\usepackage{txfonts}

\allowdisplaybreaks

\def\ls{\lesssim}
\def\gs{\gtrsim}

\def\XXint#1#2#3{{\setbox0=\hbox{$#1{#2#3}{\int}$ }
\vcenter{\hbox{$#2#3$ }}\kern-.6\wd0}}

\def\({\left(}
\def \){ \right)}

\newtheorem{theorem}{Theorem}[section]
\newtheorem{lemma}[theorem]{Lemma}
\newtheorem{corollary}[theorem]{Corollary}
\newtheorem{proposition}[theorem]{Proposition}

\theoremstyle{definition}

\renewcommand{\appendix}{\par
   \setcounter{section}{0}%
   \setcounter{subsection}{0}%
   \setcounter{subsubsection}{0}%
   \gdef\thesection{\@Alph\c@section}%
   \gdef\thesubsection{\@Alph\c@section.\@arabic\c@subsection}%
   \gdef\theHsection{\@Alph\c@section.}%
   \gdef\theHsubsection{\@Alph\c@section.\@arabic\c@subsection}%
   \csname appendixmore\endcsname
 }

\numberwithin{equation}{section}

\begin{document}

\arraycolsep=1pt

\title{\bf\Large $L^2$ Boundedness of Hilbert Transforms along Variable Flat Curves
\footnotetext{\hspace{-0.35cm} 2010 {\it
Mathematics Subject Classification}. Primary 42B20;
Secondary 42B25.
\endgraf {\it Key words and phrases.} Hilbert transform, variable flat curve.
\endgraf This work was partially supported by
 NSFC-DFG  (Grant Nos.
11761131002).}}
\author{Junfeng Li and Haixia Yu\footnote{Corresponding author.}}
\date{}
\maketitle

\vspace{-0.7cm}

\begin{center}
\begin{minipage}{13cm}
{\small {\bf Abstract}\quad In this paper, the $L^2$ boundedness of the Hilbert transform along variable flat curve $(t,P(x_1)\gamma(t))$
$$H_{P,\gamma}f(x_1,x_2):=\mathrm{p.\,v.}\int_{-\infty}^{\infty}f(x_1-t,x_2-P(x_1)\gamma(t))\,\frac{\textrm{d}t}{t},\quad \forall\, (x_1,x_2)\in\mathbb{R}^2,$$
is studied, where $P$ is a real polynomial on $\mathbb{R}$. A new sufficient condition on the curve $\gamma$ is introduced.
}
\end{minipage}
\end{center}


\section{Introduction}

In this paper we consider the Hilbert transform along variable flat curve $(t,P(x_1)\gamma(t))$
\begin{equation}\label{Hilbert transform} H_{P,\gamma}f(x_1,x_2):=\mathrm{p.\,v.}\int_{-\infty}^{\infty}f(x_1-t,x_2-P(x_1)\gamma(t))\,\frac{\textrm{d}t}{t},\quad \forall\, (x_1,x_2)\in\mathbb{R}^2,
\end{equation}
where $P:\ \mathbb{R} \rightarrow \mathbb{R}$ is a real polynomial. Bennett in \cite{BJ} pointed out that the prime interest in the general curve study for the operator above is to include some $\gamma$ which vanish to infinite order at the origin. One will see that the curve \eqref{eq:ly} is such kind of curve and satisfies all of conditions in Theorem \ref{main result}. Therefore, our results satisfy Bennett's concerns. The interest in this operator can be traced back to the outstanding Stein conjecture. For any measurable map $v:\ \mathbb R^2\rightarrow \left\{x\in \mathbb R^2:\ |x|=1\right\}$ and any Schwartz function $f$ on $\mathbb R^2$, define
\begin{equation}\label{H}
H_{v,\varepsilon}f(x):=\mathrm{p.\,v.}\int_{-\varepsilon}^{\varepsilon} f(x-v(x)t)\,\frac{\textrm{d}t}{t},\quad \forall\, x\in\mathbb{R}^2.
\end{equation}
In \cite{Stein}, Stein conjectured that $H_{v,\varepsilon}$ maps $L^2$ into weak $L^2$ whenever $v$ is a Lipschitz function with $\|v\|_{\textrm{Lip}}\approx\epsilon^{-1}$. Christ et al. in \cite{CNSW} established the $L^p$ boundedness of $H_{v,\epsilon}$ under the condition that $v$ is $C^\infty$ with extra curvature condition, where $p\in(1,\infty)$. To connect our operator with $H_{v,\epsilon}$ in \eqref{H}, let us ignore the cut off and consider
  $$H_vf(x):=\mathrm{p.\,v.}\int_{-\infty}^{\infty}f(x-v(x)t)\,\frac{\textrm{d}t}{t},\quad \forall\, x\in\mathbb{R}^2.$$
In 2006, Lacey and Li \cite{LL1} considered the case that $H_v$ applied on a dyadic piece $P_kf$ with $k\in \mathbb Z$. Here $P_k$ denotes the Littlewood-Paley projector operator. They showed that $H_{v}P_k$ is $L^p$ bounded for any measurable vector $v$ and $p\in [2,\infty)$ and the bound is independent of $k$. Later, in \cite{LL2}, by asking for $C^{1+\alpha}(\alpha>0)$ smooth of the vector $v$, Lecay and Li obtained the $L^2$ boundedness of $H_{v}$. More recently, Stein and Street \cite{SS} established the $L^p$ boundedness of $H_{v}$ with analytic vector in a more general context, where $p\in (1,\infty)$. For more details of Stein conjecture we refer the reader to the nice memoir \cite{LL2}.

Beside the direct results on the Stein conjecture, a easier case is $v(x_1,x_2)=(1,u(x_1))$. Let us define
\begin{equation}\label{Hilbert o}
H_{u,\gamma}f(x_1,x_2):=\mathrm{p.\,v.}\int_{-\infty}^{\infty}f(x_1-t,x_2-u(x_1)\gamma(t))\,\frac{\textrm{d}t}{t},\quad \forall\, (x_1,x_2)\in\mathbb{R}^2.
\end{equation}
For $\gamma(t):=t$, based on Lacey and Li's works \cite{LL1,LL2}, Bateman in \cite{B1} proved that
for any measurable function $u$, $H_{u,\gamma} P_k$ is bounded on $L^p(\mathbb{R}^2)$ for any given $p\in (1,\infty)$ and uniformly for all $k\in\mathbb Z$, where $P_k$ denotes the Littlewood-Paley projection operator in the second variable, and in \cite{B2}, Bateman and Thiele obtained that $H_{u,\gamma}$ is bounded on $L^p(\mathbb{R}^2)$ for any given $p\in (\frac{3}{2},\infty)$. Moreover, let $\gamma$ be $|t|^\alpha$ or $\textrm{sgn}(t)|t|^\alpha$, $\alpha>0$, $\alpha\neq1$, Guo in \cite{G1,G2} obtained the $L^p$ boundedness of $H_{u,\gamma}$ for any given $p\in (1,\infty)$.

In this paper we consider a more special case say that $u$ is a polynomial, i.\ e. $u=P$. Then $H_{u,\gamma}$ in \eqref{Hilbert o} becomes the operator $H_{P,\gamma}$ in \eqref{Hilbert transform}. In this case, an induction on the degree of polynomial is a robust approach.  The start point of the induction is the $L^p$ boundedness of the following Hilbert transform along general curve $(t,\gamma(t))$
\begin{equation}\label{Hilbert operator}
H_{\gamma}f(x_1,x_2):=\mathrm{p.\,v.}\int_{-\infty}^{\infty}f(x_1-t,x_2-\gamma(t))\,\frac{\textrm{d}t}{t},\quad \forall\, (x_1,x_2)\in\mathbb{R}^2.
\end{equation}
This operator has independent interests, which is another motivation of this paper. A fundamental question here is to establish the $L^p$ boundedness of \eqref{Hilbert operator} under some general conditions of the curve $\gamma$. There are enumerate literatures on this problem; see, for example, \cite{BN,Chr,CR,CVWW,CZ,NVWW,VWW,W}. As Stein and Waigner pointed out in \cite{SW} that the curvature of the considered curve plays a crucial role in this project. In the same paper, Stein and Waigner showed that if $\gamma$ is well-curved\footnote{We refer the reader to P.1240 in \cite{SW} for the definition of the well-curved curve.} then $H_{\gamma}$ is bounded on $L^p(\mathbb{R}^2)$ for any given $p\in (1,\infty)$. In \cite{CCCD}, the well-curved condition was released to $\gamma\in C^{2}(0,\infty)$ is an odd or even, convex curve, $\gamma(0)=\gamma'(0)=0$, and satisfies the following double condition:
\begin{center}
There exists $\lambda\in(1,\infty)$ so that $\gamma'(\lambda t)\geq2\gamma'( t)$ for any $t\in (0,\infty)$.  \quad(D)
\end{center}
Let $h(t)=t\gamma'(t)-\gamma(t)$. In \cite{CCVWW} condition (D) was replaced by the following infinitesimally doubling condition:
\begin{center}
There exists $\varepsilon_0\in(0,\infty)$ so that $ h'(t)\geq \varepsilon_0\frac{h(t)}{t}$ for any $t\in (0,\infty)$. \quad(ID)
\end{center}
There are more general curves to guarantee the $L^p$ boundedness of \eqref{Hilbert operator} for any given $p\in (1,\infty)$, see \cite{ZS}. We satisfy ourselves to recall the above (D) or (ID).

We can now state our main result on the boundedness of $H_{P,\gamma}$ in \eqref{Hilbert transform}.
\begin{theorem}\label{main result}Let $P:\ \mathbb{R} \rightarrow \mathbb{R}$ be a real polynomial of degree $n$, and $\gamma\in C^{2}(\mathbb{R})$ be either odd or even, convex curve on $(0,\infty)$, and satisfying
\begin{enumerate}
  \item[\rm(i)] $\gamma(0)=\gamma'(0)=0$,
  \item[\rm(ii)] $\frac{\gamma''(t)}{\gamma'(t)}$ is decreasing on $(0,\infty)$,
  \item[\rm(iii)] There exists a positive constant $C_1$ such that $\frac{t\gamma''(t)}{\gamma'(t)}\geq C_1$ for any $t\in (0,\infty)$,
  \item[\rm(iv)] $\gamma''(t)$ is monotone on $(0,\infty)$.
\end{enumerate}
Then the Hilbert transform $H_{P,\gamma}$ is bounded on $L^2(\mathbb{R}^{2})$ with a bound that can be taken to be independent of the coefficients of $P$ and dependent only on $n$ and $C_1$.
\end{theorem}

Throughout this paper, we always use $C$ to denote a positive
constant, independent of the main parameters involved, but whose
value may differ from line to line. We use $C_{(n)}$ to denote a positive constant depending on the indicated parameters $n$, and also whose
value may differ from line to line. The positive constants with subscripts,
such as $C_1$, do not change in different
occurrences. For two real functions $f$ and $g$, if $f\le Cg$, we then write $f\ls g$ or $g\gs f$;
if $f\ls g\ls f$, we then write $f\approx g$.

There were some results in this topic; see, for example, \cite{CP,CSWW,RS,S}. The condition (i) of Theorem \ref{main result} states that the curve $\gamma$ is somehow flat at zero. Besides this condition, the conditions (ii),(iii) and (iv) in Theorem \ref{main result} are used to describe the curvature conditions of $\gamma$. In \cite{CWW}, Carbery et al. set up the $L^p $ boundedness of $H_{P,\gamma}$ with $P(x_1)=x_1$ for any given $p\in (1,\infty)$. Where the curvature conditions are as follows:
\begin{center}
$\frac{t\gamma''(t)}{\gamma'(t)}$ is decreasing on $(0,\infty)$ and has a positive bounded from below. \quad(CWW)
\end{center}
Under the same conditions, Bennett in \cite{BJ,BJM}, established the $L^2$ boundedness of $H_{P,\gamma}$ for general polynomial $P$. More recently, Chen and Zhu \cite{CZx} obtained the $L^2$ boundedness of $H_{P,\gamma}$ by asking the curvature conditions as
\begin{center}
$(\frac{\gamma''}{\gamma'})'(t)\leq -\frac{\lambda}{t^2}$ for any $t\in (0,\infty)$ and some positive constant $\lambda$. \quad(CZ)
\end{center}

This paper is  organized as following. We devote section 2 to make clear that our curvature conditions are not stronger than (CWW) or (CZ). Actually, the conditions (ii) and (iii) in Theorem \ref{main result} are implied by (CWW) and (CZ). Our condition (iv) in Theorem \ref{main result} is not too strong. We will give a example which does not satisfies (CWW) or (CZ) but verifies our conditions. Section 3 contains some primary lemmas which will be used in the proof of the main result. In section 4, we will give the proof of Theorem \ref{main result}.

\section{Curve $\gamma$}

Since $\gamma$ is convex on $(0,\infty)$ and belongs to $C^2(0,\infty)$, thus $$\gamma''(t)\geq0,\quad \forall\ t\in(0,\infty).$$
It means that $\gamma'$ is increasing on $(0,\infty)$. This, combined with the fact that $\gamma'(0)=0$, further implies that
$$\gamma'(t)\geq0,\quad \forall\ t\in(0,\infty).$$
Even we pursue the general curve, the homogeneous curve $\gamma(t)=t^{\alpha}, t\in(0,\infty)$ with $\alpha>1$ satisfies all the conditions in Theorem \ref{main result}, for $t\in (-\infty,0]$, $\gamma$ is given by its even or odd property. And the homogeneous curve is the model curve. Follows are some other curves $\gamma$ satisfy our conditions in Theorem \ref{main result}. we here only write the part $t\in [0,\infty)$, and which $\gamma(t)=\pm \gamma(-t)$ for $t\in (-\infty,0]$.
\begin{enumerate}
  \item[(i)] for any $t\in [0,\infty)$, $\gamma(t):=t^2\log(1+t)$,
  \item[(ii)] for any $t\in [0,\infty)$, $\gamma(t):=t^2e^{-\frac{1}{t}}$,
  \item[(iii)] for any $t\in [0,\infty)$, $\gamma(t):=\int_0^t \tau^\alpha\log(1+\tau) \,\textrm{d}\tau$, $\alpha\in[1,\infty)$,
  \item[(iv)] for any $t\in [0,\infty)$, $\gamma(t):=\int_0^t \tau^\alpha e^{-\frac{1}{\tau}} \,\textrm{d}\tau$, $\alpha\in[1,\infty)$,
  \item[(v)] for any $t\in [0,\infty)$, $\gamma(t):=\int_0^t \tau^\alpha \arctan \tau \,\textrm{d}\tau$, $\alpha\in[1,\infty)$.
\end{enumerate}

It is easy to see that (CWW) implies conditions (ii) and (iii). It is also clear that (CZ) implies condition (ii).
Actually in \cite{CZx}, instead of (CZ), Chen and Zhu used a weaker condition
$$\frac{\gamma''(t)}{\gamma'(t)}-\frac{\gamma''(s)}{\gamma'(s)}\geq \frac{\lambda (s-t)^M}{(s+t)^{M+1}},\quad \forall\ 0<t<s,\quad (\text{wCZ})$$
for some positive constants $\lambda$ and $M$. Which is also stronger than the decreasing of $\frac{\gamma''}{\gamma'}$. Mean while, condition (CZ) also implies condition (iii). To see this, let us denote
$$F(t):=\frac{\gamma''(t)}{\gamma'(t)}-\frac{\lambda}{t},\quad \forall\ t\in (0,\infty).$$
Since we know that $\frac{\gamma''}{\gamma'}\geq0$ for any $t\in (0,\infty)$ and $-\frac{\lambda}{t}\rightarrow0$ as $t\rightarrow \infty$. We have $\lim_{t\rightarrow\infty} F(t)\geq0$. By (CZ), we have
$$F'(t)=\Big(\frac{\gamma''}{\gamma'}\Big)'(t)+\frac{\lambda}{t^2}\leq0,\quad \forall\ t\in (0,\infty).$$ Thus we have
$$F(t)\geq0 \Leftrightarrow \frac{t\gamma''(t)}{\gamma'(t)}\geq \lambda,\quad \forall\ t\in (0,\infty).$$

The condition (iv) is our main contribution for this problem. It has already appeared in \cite{NW}. Nagel and Winger proved that under conditions (i),(iii) and (iv), $H_{\gamma}$ in \eqref{Hilbert operator} is bounded on $L^p(\mathbb{R}^{2})$ for $p$ is very close to 2. Actually, if $\gamma''$ is increasing on $(0,\infty)$, then (iv) implies (iii). But if $\gamma''$ is decreasing on $(0,\infty)$, then all the conditions in Theorem \ref{main result} are independent from each other. For given $x,y\in (0,\infty)$, let $0<z-y<z-x<\infty$, we set
$$\varphi(z):=\gamma''(z-x)-\gamma''(z-y).$$
Condition (iv) is used to guarantee that the sign of $\varphi$ changes finite many times on $(0,\infty)$ and the number does not depend on $x,y,z$. There are many different ways to form up this condition, and here we use condition (iv) in the proof of Theorem \ref{main result}.

Now we give an example which does not satisfies (CWW) or (CZ) but verifies our conditions (i),(ii),(iii) and (iv). Thus the conditions in Theorem \ref{main result} are not trivial.
Let
\begin{align}\label{eq:ly}
\gamma(t):=\int_0^t e^\tau e^{-\frac{1}{\tau}}\, \textrm{d}\tau,\quad \forall\ t\in [0,\infty) ,
\end{align}
and
$$\gamma(t):=\pm \gamma(-t),\quad \forall\ t\in (-\infty,0].$$
We calculate
$$\gamma'(t)= e^t e^{-\frac{1}{t}},\quad \forall\ t\in [0,\infty) ,$$
$$\gamma''(t)= e^t e^{-\frac{1}{t}}+e^t e^{-\frac{1}{t}}t^{-2}\geq 0,\quad \forall\ t\in [0,\infty) ,$$
$$\gamma'''(t)= e^t e^{-\frac{1}{t}}\left(1+2t^{-2}-2t^{-3}+t^{-4}\right)=e^t e^{-\frac{1}{t}}\frac{t^4+2((t-\frac{1}{2})^2+\frac{1}{4})}{t^4}\geq 0,\quad \forall\ t\in [0,\infty).$$
So we have $\gamma\in C^{2}(\mathbb{R})$ is either odd or even, convex curve on $(0,\infty)$, and $\gamma''$ is increasing on $(0,\infty)$, and
$$\gamma(0)=\gamma'(0)=0,$$
$$\frac{\gamma''(t)}{\gamma'(t)}=1+\frac{1}{t^2},\quad \forall\ t\in [0,\infty).$$
Therefore, $\frac{\gamma''}{\gamma'}$ is decreasing on $(0,\infty)$ and
$$\frac{t\gamma''(t)}{\gamma'(t)}=t+\frac{1}{t}\geq 2,\quad \forall\ t\in [0,\infty).$$
Then, curve $\gamma$ satisfies the conditions (i),(ii),(iii) and (iv) in Theorem \ref{main result}.

We now show that this curve does not satisfies (CWW). We set
$$\lambda(t):=\frac{t\gamma''(t)}{\gamma'(t)}=t+\frac{1}{t},\quad \forall\ t\in (0,\infty).$$
It is easy to get that
$$\lambda'(t)=1-\frac{1}{t^2},\quad \forall\ t\in (0,\infty).$$
Thus, $\lambda$ is not decreasing on $(0,\infty)$.

For condition (CZ) , we have
$$\left(\frac{\gamma''}{\gamma'}\right)'(t)=-2 t^{-3},\quad \forall\ t\in (0,\infty).$$
Thus, there is no positive number $\lambda$ such that (CZ) is true.

Even for the weaker condition (wCZ). Let $s=2t$, if $\gamma$ satisfies (wCZ), then
$$\frac{1}{t^2}-\frac{1}{4t^2}\geq  \frac{\lambda t^M}{(3t)^{M+1}},\quad \forall\ t\in (0,\infty),$$
thus
$$\frac{3}{4t}\geq \frac{\lambda}{3^{M+1}},\quad \forall\ t\in (0,\infty).$$
But such positive constants $\lambda$ and $M$ can not exist as $t\rightarrow \infty$.

Next, we try to give a geometric explanation of the conditions (ii) and (iii). Condition (ii) is equals to that $\frac{\gamma'}{\gamma''}$ is increasing on $(0,\infty)$. See Figure 1, $A(t-\frac{\gamma'(t)}{\gamma''(t)},0)$ is the intersection between $x$-axis and the tangent line to the curve $(x,y)=(t,\gamma'(t))$, and point $B(t,0)$ is the intersection between $x$-axis and the vertical line to the curve $(x,y)=(t,\gamma'(t))$, then the quantity $\frac{\gamma'}{\gamma''}$ is the distance between point $A$ and point $B$. Condition (ii) said that the distance between $AB$ will grow large as $t$ goes to right. But condition (iii) said that the distance is always less than $\frac{t}{C_1}$. Here $C_1$ is the constant in condition (iii). If we simply set $C_1\geq 1$ then it means that $\gamma'$ is also convex on $(0,\infty)$.

\begin{figure}[htbp]
  \centering
  \includegraphics[width=3.5in]{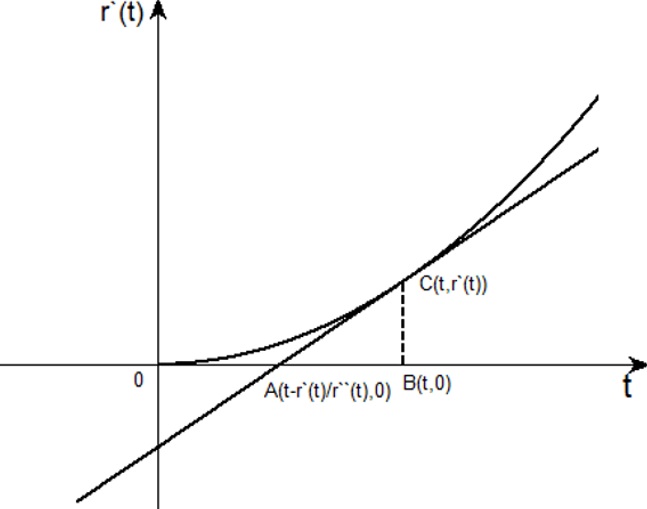}\\
  \caption{$|AB|=\frac{\gamma'(t)}{\gamma''(t)}$}\label{Figure£º1}
\end{figure}

Actually, if $\gamma'$ is also convex on $(0,\infty)$, then we have $\gamma''$ is increasing on $(0,\infty)$. Thus condition (iii) can also be true with $C_1=1$. Thus we have the following corollary:
\begin{corollary}Let $P:\ \mathbb{R} \rightarrow \mathbb{R}$ be a real polynomial of degree $n$, and $\gamma\in C^{2}(\mathbb{R})$ be either odd or even, convex curve on $(0,\infty)$, and satisfying
\begin{enumerate}
  \item $\gamma(0)=\gamma'(0)=0$,
  \item $\frac{\gamma''(t)}{\gamma'(t)}$ is decreasing on $(0,\infty)$,
  \item $\gamma'(t)$ is also convex on $(0,\infty)$.
\end{enumerate}
Then the Hilbert transform $H_{P,\gamma}$ is bounded on $L^2(\mathbb{R}^{2})$ with a bound that can be taken to be independent of the coefficients of $P$ and dependent only on $n$.
\end{corollary}

\section{Some Lemmas}

In this section, we collect several lemmas which will be used in the proof of Theorem \ref{main result}.

\begin{lemma}\label{le2.1} The conditions of Theorem \ref{main result} imply the double condition $(\textrm{D})$ and infinitesimally doubling condition $(\textrm{ID})$.
\end{lemma}

\begin{proof} For double condition $(\textrm{D})$, note that $\gamma(0)=\gamma'(0)=0$, $\gamma'$ is increasing on $(0,\infty)$, and there exists a positive constant $C_1$ such that $\frac{t\gamma''(t)}{\gamma'(t)}\geq C_1$ for any $t\in (0,\infty)$. Let $\lambda:=e^{\frac{2}{C_1}}$, then

$$\gamma'(\lambda t)=\int_0^{\lambda t}\gamma''(\tau)\textrm{d}\tau \geq \int_t^{\lambda t}\gamma''(\tau)\textrm{d}\tau
   \geq C_1\int_t^{\lambda t} \frac{\gamma'(\tau)}{\tau}\textrm{d}\tau \geq  C_1\gamma'(t) \log \lambda=2\gamma'(t),\quad \forall\ t\in (0,\infty).$$

To verify the infinitesimally double condition $(\textrm{ID})$, let $\varepsilon_0:=C_1$, then
$$\frac{h'(t)}{\gamma'(t)}= \frac{t\gamma''(t)}{\gamma'(t)}\geq C_1\geq C_1 \frac{t\gamma'(t)-\gamma(t)}{t\gamma'(t)}=C_1 \frac{h(t)}{t\gamma'(t)},\quad \forall\ t\in (0,\infty),$$
thus
$$h'(t)\geq \varepsilon_0 \frac{h(t)}{t},\quad \forall\ t\in (0,\infty).$$
\end{proof}

\begin{lemma}\label{le2.2} Let $\gamma$ be the same as in Theorems \ref{main result}. Let $x,y\in \mathbb{R}$ be fixed. If $1\leq z-y < z-x \leq 2$, then $\Upsilon(z):=\frac{\gamma(\omega 2^k (z-x))-\gamma(\omega 2^k(z-y))}{\omega 2^k\gamma'(\omega 2^k(z-x))-\omega 2^k\gamma'(\omega 2^k(z-y))}$ is increasing on this domain, and $ \Upsilon(z) \leq \frac{2}{C_1}$ uniformly in $\omega$ and $k$ on this domain, where $\omega \in \mathbb{R}$, $\omega>0$ and $k\in \mathbb{Z}$, $k\geq 0$ .
\end{lemma}

\begin{proof} Let $g(s,t):=\frac{\gamma(s)-\gamma(t)}{\gamma'(s)-\gamma'(t)}$ for any $s> t> 0$. By generalized mean value theorem and note that $\frac{\gamma''}{\gamma'}$ is decreasing on $(0,\infty)$, it is easy to see that $\frac{\partial g}{\partial s}\geq 0$ and $\frac{\partial g}{\partial t}\geq 0$. Thus, for fixed $x,y\in \mathbb{R}$, for any $z$ satisfies $1\leq z-y < z-x \leq 2$, we have
$\Upsilon(z)=\frac{1}{\omega 2^k}g(\omega 2^k (z-x),\omega 2^k(z-y))$. Therefore, we get $$\Upsilon'(z)=\frac{\partial g}{\partial s}(\omega 2^k (z-x),\omega 2^k(z-y))+\frac{\partial g}{\partial t}(\omega 2^k (z-x),\omega 2^k(z-y))\geq 0.$$
Thus, $\Upsilon$ is increasing uniformly in $\omega$ and $k$ on this domain.

By the generalised mean value theorem, and the fact that $\frac{t\gamma''(t)}{\gamma'(t)}\geq C_1$ for any $t\in (0,\infty)$, we obtain
$$\Upsilon(z)=\frac{\gamma'(\omega 2^k\theta_1)}{\omega 2^k\gamma''(\omega 2^k\theta_1)}=\theta_1\frac{\gamma'(\omega 2^k\theta_1)}{\omega 2^k\theta_1\gamma''(\omega 2^k\theta_1)}\leq \frac{2}{C_1}$$
on this domain, where $1\leq z-y\leq \theta_1 \leq z-x\leq 2$.
\end{proof}

\begin{lemma}\label{le2.3}(Lemma 13 in \cite{BJ}) Let $P$ be a real monic polynomial of degree $n$ $(n>0)$ and of one real variable. Let $U$ be the union of the set of roots of $P'$ and of $P''$ over $\mathbb{R}$. There exists $C>0$ depending only on $n$, such that if $\textrm{dist}(z, U)>\delta$, then $|P'(z)|\geq C \delta^{n-1} $ for any $\delta>0$.
 \end{lemma}

\begin{lemma}\label{le2.4}Let $P$ be a real monic polynomial of degree $n\, (n>0)$ and of one real variable, and
$$E_k:=\left\{x\in \mathbb{R}:\ \left | \frac{\tilde{P}_k(x)}{\tilde{P}_k'(x)}\right |\leq \frac{4}{C_1} \textrm{and} \left(\frac{\tilde{P}_k}{\tilde{P}_k'}\right)'(x)\leq \frac{1}{8n} \right \},$$
where $\tilde{P}_k(x):=2^{-nk} P(2^kx)$, $k\geq 0$, $k\in \mathbb{Z}$. Then for any $\alpha\in (0,1)$, we have $ \Sigma_{k\geq 0}|E_k|^\alpha \leq C $, where $C$ only depends on $n$ and $C_1$.
\end{lemma}

\begin{proof}We denote the roots of $P$ as $\{\nu_j\}_{j=1}^{m}\subset \mathbb{R}$ and $\{\beta_j\}_{j=1}^{n'}\cup \{\bar{\beta_j}\}_{j=1}^{n'} \subset \mathbb{C}\backslash \mathbb{R}$, where $0\leq m\leq n, n'=\frac{n-m}{2}$ and $\beta_j=a_j+ib_j$. Then
$$\tilde{P}_k(x)=\Pi_{j=1}^m (x-2^{-k}\nu_j)  \Pi_{j=1}^{n'} (x-2^{-k}\beta_j) (x-2^{-k}\bar{\beta_j}).$$
Thus
$$\frac{\tilde{P}_k'(x)}{\tilde{P}_k(x)}=\sum_{j=1}^m \frac{1}{x-2^{-k}\nu_j}+2  \sum_{j=1}^{n'}  \frac{x-2^{-k}a_j}{(x-2^{-k}a_j)^2+(2^{-k}b_j)^2}$$
and
$$\left(\frac{\tilde{P}_k'}{\tilde{P}_k}\right)'(x)=-\sum_{j=1}^m \frac{1}{(x-2^{-k}\nu_j)^2}-2  \sum_{j=1}^{n'}  \frac{(x-2^{-k}a_j)^2-(2^{-k}b_j)^2}{((x-2^{-k}a_j)^2+(2^{-k}b_j)^2)^2}.$$
For any $l\in \mathbb Z$, we set
$$E_{k,l}:=\left\{x\in \mathbb{R}:\ 2^{l-1}\frac{4}{C_1}\leq\left | \frac{\tilde{P}_k(x)}{\tilde{P}_k'(x)}\right |\leq 2^l\frac{4}{C_1} \textrm{and} \left(\frac{\tilde{P}_k}{\tilde{P}_k'}\right)'(x)\leq \frac{1}{8n} \right \}.$$
Hence,
\begin{align}\label{eq:2.0}
|E_k| = \sum_{l=-\infty}^{0}|E_{k,l}|.
\end{align}
Note that $\left(\frac{\tilde{P}_k}{\tilde{P}_k'}\right)'(x)=-\left(\frac{\tilde{P}_k}{\tilde{P}_k'}\right)^2(x) \left(\frac{\tilde{P}_k'}{\tilde{P}_k}\right)'(x) $, for any $x\in E_{k,l}$, by definition
\begin{align}\label{eq:2.1}
\left(\frac{\tilde{P}_k}{\tilde{P}_k'}\right)^2(x)\left(\sum_{j=1}^m \frac{1}{(x-2^{-k}\nu_j)^2}+2  \sum_{j=1}^{n'}  \frac{(x-2^{-k}a_j)^2-(2^{-k}b_j)^2}{((x-2^{-k}a_j)^2+(2^{-k}b_j)^2)^2}\right)\leq \frac{1}{8n}
\end{align}
and
\begin{align}\label{eq:2.2}
\left | \sum_{j=1}^m \frac{1}{x-2^{-k}\nu_j}+2  \sum_{j=1}^{n'}  \frac{x-2^{-k}a_j}{(x-2^{-k}a_j)^2+(2^{-k}b_j)^2}\right |\geq \frac{C_1}{4\cdot2^l}. \end{align}
From which we obtain
\begin{align}\label{eq:2.3}
\sum_{j=1}^m \frac{1}{|x-2^{-k}\nu_j|}+2  \sum_{j=1}^{n'}  \frac{|x-2^{-k}a_j|}{(x-2^{-k}a_j)^2+(2^{-k}b_j)^2}\geq \frac{C_1}{4\cdot2^l}.
\end{align}
By $\textrm{H}\ddot{\textrm{o}}\textrm{lder's}$ inquality, \eqref{eq:2.3} implies
\begin{align}\label{eq:2.4}
\sum_{j=1}^m \frac{1}{(x-2^{-k}\nu_j)^2}+2  \sum_{j=1}^{n'}  \frac{(x-2^{-k}a_j)^2}{((x-2^{-k}a_j)^2+(2^{-k}b_j)^2)^2}\geq \frac{1}{n} \left(\frac{C_1}{4\cdot2^l}\right)^2 .
\end{align}
From \eqref{eq:2.1}, we have
\begin{align}\label{eq:2.5}
\sum_{j=1}^m \frac{1}{(x-2^{-k}\nu_j)^2}+2  \sum_{j=1}^{n'}  \frac{(x-2^{-k}a_j)^2-(2^{-k}b_j)^2}{((x-2^{-k}a_j)^2+(2^{-k}b_j)^2)^2}\leq \frac{1}{8n} \left(\frac{C_1}{2\cdot2^l}\right)^2.
\end{align}
Combining \eqref{eq:2.4} and \eqref{eq:2.5}, we have
\begin{align}\label{eq:2.6}
\sum_{j=1}^{n'}  \frac{(2^{-k}b_j)^2}{((x-2^{-k}a_j)^2+(2^{-k}b_j)^2)^2}
 \geq\frac{1}{2} \left(\frac{1}{n} \left(\frac{C_1}{4\cdot2^l}\right)^2-\frac{1}{8n} \left(\frac{C_1}{2\cdot2^l}\right)^2\right)=\frac{1}{64n} \left(\frac{C_1}{2^l}\right)^2.
\end{align}
By Minkowski inequality, it gives
\begin{align}\label{eq:2.7}
 \sum_{j=1}^{n'}    \frac{|2^{-k}b_j|}{(x-2^{-k}a_j)^2+(2^{-k}b_j)^2}\geq \frac{1}{8 n^{\frac{1}{2}}}\frac{C_1}{2^l}.
 \end{align}
By pigeonholing, there exists $1\leq j_0 \leq n'$ such that
\begin{align}\label{eq:2.8}
\frac{|2^{-k}b_{j_0}|}{(x-2^{-k}a_{j_0})^2+(2^{-k}b_{j_0})^2}\geq \frac{1}{8n' n^{\frac{1}{2}}}\frac{C_1}{2^l} .
\end{align}
It is then not hard to check that
$$E_{k,l} \subset \bigcup_{j=1}^{n'}E_{k,l,j},$$
where
\begin{align}
 E_{k,l,j}\nonumber
 =\left\{x\in \mathbb{R}:\  |2^{-k}b_{j}|\leq  \frac{8n' n^{\frac{1}{2}}2^l }{C_1} ~\textrm{and}~ |x-2^{-k}a_{j}| \leq  \left(\frac{8n' n^{\frac{1}{2}}2^l }{C_1}\right)^{\frac{1}{2}}  |2^{-k}b_{j}|^{\frac{1}{2}} \right \}.\
 \end{align}
Since $E_{k,l,j}=\emptyset$ if $|2^{-k}b_{j}|\geq  \frac{8n' n^{\frac{1}{2}}2^l }{C_1} $, and $|E_{k,l,j}| \leq \left(\frac{8n' n^{\frac{1}{2}}2^l }{C_1}\right)^{\frac{1}{2}}  |2^{-k}b_{j}|^{\frac{1}{2}} $ if $|2^{-k}b_{j}|\leq  \frac{8n' n^{\frac{1}{2}}2^l }{C_1} $, thus
$$ \sum_{k\geq 0} |E_{k,l,j}|^\alpha\lesssim 2^{\alpha l}$$
uniformly in $j$ and $l$. Hence, for any $\alpha\in (0,1)$,
\begin{eqnarray*}
 \sum_{k\geq 0}|E_k|^\alpha
    \leq   \sum_{k\geq 0}\left(\sum_{l=-\infty}^{0}\sum_{j=1}^{n'}|E_{k,l,j}|\right)^\alpha
      \leq   \sum_{l=-\infty}^{0}\sum_{j=1}^{n}\sum_{k\geq 0}|E_{k,l,j}|^\alpha
     \lesssim   \sum_{l=-\infty}^{0}2^{\alpha l}
    \lesssim 1.
\end{eqnarray*}

\end{proof}
It is easy to see that
\begin{lemma}\label{le2.5} Let $f\in C(\mathbb{R})$, suppose that the sign of $f'$ changes $m$ times on $\mathbb{R}$ and there exists a positive constant $C$ such that $|f(z)|\leq C$ for any $z\in \mathbb{R}$. Then $\int_{-\infty}^{\infty} |f'(z)| \,\textrm{d}z\leq 2(m+1)C$.
\end{lemma}

\section{Proof of the main result}

We devote this section to prove Theorem \ref{main result}. By Fourier transform and Plancherel's formula, see \cite{PS}, we have
$$\|H_{P,\gamma}\|_{L^2(\mathbb{R}^{2})\rightarrow L^2(\mathbb{R}^{2})} \leq \sup_{u\in \mathbb{R}} \|S_{u}\|_{L^2(\mathbb{R})\rightarrow L^2(\mathbb{R})} ,$$
where
\begin{equation}\label{4.1}S_uf(x):=\mathrm{p.\,v.}\int_{-\infty}^{\infty}e^{-iuP(x)\gamma(y)}  f(x-y)\,\frac{\textrm{d}y}{y},\quad \forall\, x\in\mathbb{R}^2.\end{equation}
Here $P$ is a polynomial of degree $n$. Thus, the remainder of the proof is devoted to the proof of the following Proposition:
\begin{proposition}\label{pro:4.2} Let $S_u$ be defined as in \eqref{4.1}, we have
\begin{equation}
 \|S_{u}\|_{L^2(\mathbb{R})\rightarrow L^2(\mathbb{R})}\leq C.
\end{equation}
Here $C$ is independent of the coefficients of polynomial $P$ and $u$, it is dependent only on $n$ and $C_1$.
\end{proposition}
\begin{proof}
 As in \cite{BJ,BJM,CWW,CZx}, we will induct on the degree of the polynomial. The start point is the case $n=0$.  The $L^2$ boundedness of $S_u$  then can be converted to the $L^2$ boundedness of the following directional Hilbert transform $H_{\lambda,\gamma}$ along a general curve $\gamma$ defined for a fixed direction $(1,\lambda)$ for any $\lambda\in \mathbb{R}$ as
\begin{equation}
H_{\lambda,\gamma} f(x_1,x_2):=\mathrm{p.\,v.}\int_{-\infty}^{\infty} f(x_1-t,x_2-\lambda\gamma(t))\,\frac{\textrm{d}t}{t},\quad \forall\, (x_1,x_2)\in\mathbb{R}^2.\nonumber
\end{equation}
By scaling on the second variation we know that the $L^2$ boundedness of $H_{\lambda,\gamma}$, with $\lambda\neq0$, is the same as
\begin{equation}
H_\gamma f(x_1,x_2)=\mathrm{p.\,v.}\int_{-\infty}^{\infty} f(x_1-t,x_2-\gamma(t))\,\frac{\textrm{d}t}{t},\quad \forall\, (x_1,x_2)\in\mathbb{R}^2,\nonumber
\end{equation}
which has been defined in section 1. Since the $L^2$ boundedness of $H_{0,\gamma}$ is trivial, thus we could obtain an uniform $L^2$ estimate for $S_u$ if we obtain the $L^2$ boundedness of $H_\gamma$.
As we pointed out in Lemma \ref{le2.1}, the conditions of Theorem \ref{main result} imply the double condition $(\textrm{D})$ and infinitesimally doubling condition $(\textrm{ID})$, by apply the result in \cite{CCCD} or in \cite{CCVWW}, we obtain the result for the case $n=0$.

We now consider the case $n>0$. Suppose $s$ is the coefficient of the term of the highest order in $P$, since $\gamma(0)=0$ and $\gamma$ is increasing on $(0,\infty)$, we can take a positive constant $\omega$, unless $\gamma\equiv 0$ which is trivial, such that
\begin{equation}\label{4.5}|-us|\omega^n\gamma(\omega)=1.\end{equation}
Thus
\begin{eqnarray*}
 S_uf(\omega x)
   &=&\mathrm{p.\,v.}\int_{-\infty}^{\infty} e^{-iuP(\omega x)\gamma(\omega x-y)}  f(y)\,\frac{\textrm{d}y}{\omega x-y}\\
    &=&\mathrm{p.\,v.}\int_{-\infty}^{\infty} e^{-iuP(\omega x)\gamma(\omega (x-y))}  f(\omega y)\,\frac{\textrm{d}y}{ x-y}\\
   &=&\mathrm{p.\,v.}\int_{-\infty}^{\infty} e^{-iu\gamma (\omega)P(\omega x)\frac{\gamma(\omega (x-y))}{\gamma (\omega)}}  f(\omega y)\,\frac{\textrm{d}y}{ x-y}.
\end{eqnarray*}
We define
\begin{equation}\label{eq:4.3}Sf(x):=\mathrm{p.\,v.}\int_{-\infty}^{\infty} e^{iP(x)\tilde{\gamma}(x-y)}  f(y)\,\frac{\textrm{d}y}{x-y},
\end{equation}
where $P$ is a monic polynomial, $\tilde{\gamma}(t):=\frac{\gamma(\omega t)}{\gamma (\omega)}$. By scaling, we need then to set up
\begin{equation}
\|S\|_{L^2(\mathbb{R})\rightarrow L^2(\mathbb{R})}\lesssim 1.
\end{equation}

Suppose that $S$ is bounded on $L^2(\mathbb{R})$ for all polynomials $P$ of degree less than $n$, with a bound independent of the coefficients of $P$.  We decompose
\begin{align}\label{eq:S}
Sf(x)
   =& \int_{|x-y|\leq 1}  e^{iP(x)\tilde{\gamma}(x-y)} f(y)\,\frac{\textrm{d}y}{x-y}+  \sum_{k\geq 0} \int_{2^k\leq|x-y|\leq 2^{k+1}}  e^{iP(x)\tilde{\gamma}(x-y)}  f(y)\,\frac{\textrm{d}y}{x-y}\\
  =:& S^1f(x)+ \sum_{k\geq 0}S_kf(x).\nonumber
\end{align}

The first part is easy. It is exactly as the local part in \cite{BJ}, where the author asked for $\tilde{\gamma}$ is convex on $(0,\infty)$, $\tilde{\gamma} (0)=0$, $\tilde{\gamma}(1)=1$, and the inductive hypothesis. Thus
$$\|S^1\|_{L^2(\mathbb{R})\rightarrow L^2(\mathbb{R})} \lesssim 1$$
with the bound independent of the coefficients of $P$.

For $k\geq 0$, we set
$$ \tilde{S}_kf(x):=\int_{1\leq|x-y|\leq 2} e^{i2^{nk} \frac{\gamma (\omega 2^k)}{\gamma (\omega)} \tilde{P}_k(x) \frac{\gamma (\omega 2^k(x-y))}{\gamma (\omega 2^k)}  }  f(y)\,\frac{\textrm{d}y}{x-y}, $$
where $\tilde{P}_k(x):=2^{-nk} P(2^kx)$ is a real monic polynomial. $\tilde{S}_k$ is a scaling of $S_k$ and thus share the same $L^p$ norm for any given $p\in (1,\infty)$. Since $\gamma$ is either even or odd, we here only consider half operator
$$ \tilde{\mathbb{S}}_k f(x):=\int_{1\leq x-y\leq 2} e^{i2^{nk} \frac{\gamma (\omega 2^k)}{\gamma (\omega)} \tilde{P}_k(x) \frac{\gamma (\omega 2^k(x-y))}{\gamma (\omega 2^k)}  }  f(y)\,\frac{\textrm{d}y}{x-y}. $$

To distinct the critical points of the phase function, let
$$E_k:=\left\{x\in \mathbb{R}:\ \left | \frac{\tilde{P}_k(x)}{\tilde{P}_k'(x)}\right |\leq \frac{4}{C_1} \textrm{ and } \left(\frac{\tilde{P}_k}{\tilde{P}_k'}\right)'(x)\leq \frac{1}{8n} \right \}.$$
By this set, we decompose $\tilde{\mathbb{S}}_k$ further as $\tilde{\mathbb{S}}_{ka}$ where the phase function has critical points and $\tilde{\mathbb{S}}_{kb}$ where the phase function has not critical points.
$$\tilde{\mathbb{S}}_{ka}f(x):= \chi_{E_k}(x)\int_{1\leq x-y\leq 2} e^{i2^{nk} \frac{\gamma (\omega 2^k)}{\gamma (\omega)} \tilde{P}_k(x) \frac{\gamma (\omega 2^k(x-y))}{\gamma (\omega 2^k)}  }  f(y)\,\frac{\textrm{d}y}{x-y}$$
and
$$\tilde{\mathbb{S}}_{kb}f(x):= \chi_{E_k^{\complement}}(x)\int_{1\leq x-y\leq 2} e^{i2^{nk} \frac{\gamma (\omega 2^k)}{\gamma (\omega)} \tilde{P}_k(x) \frac{\gamma (\omega 2^k(x-y))}{\gamma (\omega 2^k)}  }  f(y)\,\frac{\textrm{d}y}{x-y}.$$

To estimation for $\tilde{\mathbb{S}}_{ka}$, we use the character that the set $E_k$ is not very large. It is easy to see that
$$ \|\tilde{\mathbb{S}}_{ka}f\|_{L^1(\mathbb{R})} \leq |E_k| \|f\|_{L^1(\mathbb{R})}$$
and
$$ \|\tilde{\mathbb{S}}_{ka}f\|_{L^\infty(\mathbb{R})} \lesssim \|f\|_{L^\infty(\mathbb{R})} .$$
By interpolation,
$$ \|\tilde{\mathbb{S}}_{ka}f\|_{L^p(\mathbb{R})} \leq |E_k|^{\frac{1}{p}} \|f\|_{L^p(\mathbb{R})}.$$
Thus by Lemma \ref{le2.4},
 \begin{equation}\label{4.7}\sum_{k\geq 0}\|\tilde{\mathbb{S}}_{ka}\|_{L^p(\mathbb{R})\rightarrow L^p(\mathbb{R})} \lesssim\sum_{k\geq 0} |E_k|^{\frac{1}{p}} \lesssim 1. \end{equation}

For $\tilde{\mathbb{S}}_{kb}$, the bad part is that the size of $E_k^{\complement}$ is large but the good news is that there is no critical point of phase function. Thus we run a $TT^*$ argument. It is easy to see that \begin{equation}\label{4.8}\|\tilde{\mathbb{S}}_{k b}\|_{L^2(\mathbb{R})\rightarrow L^2(\mathbb{R})}=\|\tilde{\mathbb{S}}_{k b}^*\tilde{\mathbb{S}}_{k b}\|^{\frac{1}{2}}_{L^2(\mathbb{R})\rightarrow L^2(\mathbb{R})}.\end{equation} The kernel of $\tilde{\mathbb{S}}_{k b}^*\tilde{\mathbb{S}}_{k b}$ can be written as
$$L_k(x,y):=\int_{1 \leq z-x, z-y \leq 2,z \notin  E_k} \frac{e^{i2^{nk} \frac{\gamma (\omega 2^k)}{\gamma (\omega)} \tilde{P}_k(z) \left(\frac{\gamma (\omega 2^k(z-x))}{\gamma (\omega 2^k)} -\frac{\gamma (\omega 2^k(z-y))}{\gamma (\omega 2^k)} \right)}}{(z-x)(z-y)}\,\textrm{d}z. $$
By symmetry, it suffices to consider the kernel
$$\mathbb{L}_k(x,y):=\int_{1 \leq z-y < z-x \leq 2,z \notin  E_k}  \frac{e^{i2^{nk} \frac{\gamma (\omega 2^k)}{\gamma (\omega)} \tilde{P}_k(z) \left(\frac{\gamma (\omega 2^k(z-x))}{\gamma (\omega 2^k)} -\frac{\gamma (\omega 2^k(z-y))}{\gamma (\omega 2^k)} \right)}}{(z-x)(z-y)}\,\textrm{d}z. $$
For fixed $x,y\in \mathbb{R}$, and for any $z\in \mathbb{R}$ satisfies $1\leq z-y < z-x \leq 2$, let
$$\phi(z):=2^{nk} \frac{\gamma (\omega 2^k)}{\gamma (\omega)}  \left(\frac{\gamma (\omega 2^k(z-x))}{\gamma (\omega 2^k)} -\frac{\gamma (\omega 2^k(z-y))}{\gamma (\omega 2^k)} \right),$$
$$\psi(z):=\phi(z)\tilde{P}_k(z),$$
and for any $r\in \mathbb{R}$ satisfies $1\leq r-y < r-x \leq 2$, let
\begin{equation}\label{4.9}J^r:=\int_{(y+1,r)\backslash E_k} e^{i\psi(z)}\,\textrm{d}z. \end{equation}
For \eqref{4.9}, we collect two estimates for it in Proposition \ref{pro:4.1} which we will give the proof later. We note that $\frac{\textrm{d}}{\textrm{d}z} \frac{1}{(z-x)(z-y)} <0 $ for any $z$ satisfies $1 \leq z-y < z-x \leq 2$ and $E_k$ is made up by $C_{(n)}$ intervals. If we have the estimate of \eqref{eq:4.2}, then
\begin{align}\label{eq:4.4}
&\mathbb{L}_k(x,y)\\
 =&\int_{1 \leq z-y < z-x \leq 2, z \notin  E_k}   \frac{\textrm{d}}{\textrm{d}z} \left(J^z \right)    \frac{1}{(z-x)(z-y)}\,\textrm{d}z \nonumber\\
=&\left[  \frac{J^z}{(z-x)(z-y)} \right]_{1 \leq z-y < z-x \leq 2, z \notin  E_k}-\int_{1 \leq z-y < z-x \leq 2, z \notin  E_k}J^z \frac{\textrm{d}}{\textrm{d}z}\left ( \frac{1}{(z-x)(z-y)} \right) \,\textrm{d}z   \nonumber\\
\lesssim& \left( \sup_{z:\ 1 \leq z-y < z-x \leq 2}|J^z| \right) \frac{1}{|x-y|} + \left( \sup_{z:\ 1 \leq z-y < z-x \leq 2}|J^z| \right)     \int_{1 \leq z-y < z-x \leq 2, z \notin  E_k} \left| \frac{\textrm{d}}{\textrm{d}z}\left ( \frac{1}{(z-x)(z-y)} \right)\right|\, \textrm{d}z \nonumber\\
\approx& \left( \sup_{z:\ 1 \leq z-y < z-x \leq 2}|J^z| \right) \frac{1}{|x-y|} + \left( \sup_{z:\ 1 \leq z-y < z-x \leq 2}|J^z| \right)   \left|  \int_{1 \leq z-y < z-x \leq 2, z \notin  E_k}  \frac{\textrm{d}}{\textrm{d}z}\left ( \frac{1}{(z-x)(z-y)} \right)\, \textrm{d}z \right|\nonumber\\
\lesssim&  \left( \sup_{z:\ 1 \leq z-y < z-x \leq 2}|J^z| \right) \frac{1}{|x-y|}\nonumber\\
\lesssim&  \left(\frac{1}{|x-y|}\right)^{\frac{n+1}{n}} 2^{-k}.\nonumber
\end{align}
Meanwhile, we have the following trivial estimate,
\begin{align}\label{eq:4.5}
|\mathbb{L}_k(x,y)|\leq \int_{1 \leq z-y < z-x \leq 2,z \notin  E_k} \frac{1}{|(z-x)(z-y)|}\,\textrm{d}z\lesssim 1.
\end{align}
Therefore,
\begin{align}\label{eq:4.6}
|\mathbb{L}_k(x,y)|\lesssim \min\left\{ \left(\frac{1}{|x-y|}\right)^{\frac{n+1}{n}}2^{-k}, 1 \right\}\lesssim \left(\frac{1}{|x-y|}\right)^{\frac{1}{2}}2^{-\frac{n}{2(n+1)}k}.
\end{align}
From \eqref{eq:4.5} and \eqref{eq:4.6},
\begin{align}\label{eq:4.7}
\int_{-\infty}^{\infty}\left|\mathbb{L}_k(x,y) \right|\,\textrm{d}x
\lesssim& \int_{|x-y|\leq1} \left(\frac{1}{|x-y|}\right)^{\frac{1}{2}}2^{-\frac{n}{2(n+1)}k} \,\textrm{d}x+ \int_{|x-y|\geq1} \left(\frac{1}{|x-y|}\right)^{\frac{n+1}{n}} 2^{-k}\,\textrm{d}x\\
\lesssim&2^{-\frac{n}{2(n+1)}k}.\nonumber
\end{align}

If we have the estimate \eqref{eq:4.3}, again by the same argument as \eqref{eq:4.4} we obtain
\begin{align}\label{eq:4.8}
\mathbb{L}_k(x,y)\lesssim \left(\frac{1}{|x-y|}\right)^{\frac{n+2}{n+1}} 2^{-\frac{n}{n+1}k}.
\end{align}
Combine \eqref{eq:4.5}, which leads to
\begin{align}\label{eq:4.9}
|\mathbb{L}_k(x,y)|\lesssim \min\left\{ \left(\frac{1}{|x-y|}\right)^{\frac{n+2}{n+1}} 2^{-\frac{n}{n+1}k}, 1  \right\}\lesssim \left(\frac{1}{|x-y|}\right)^{\frac{1}{2}}2^{-\frac{n}{2(n+2)}k} ,
\end{align}
From \eqref{eq:4.8} and \eqref{eq:4.9},
\begin{align}\label{eq:4.10}
\int_{-\infty}^{\infty}\left|\mathbb{L}_k(x,y) \right|\,\textrm{d}x
\lesssim&\int_{|x-y|\leq1} \left(\frac{1}{|x-y|}\right)^{\frac{1}{2}}2^{-\frac{n}{2(n+2)}k}\, \textrm{d}x+ \int_{|x-y|\geq1} \left(\frac{1}{|x-y|}\right)^{\frac{n+2}{n+1}} 2^{-\frac{n}{n+1}k}\,\textrm{d}x\\
\lesssim&2^{-\frac{n}{2(n+2)}k}.\nonumber
\end{align}

Both \eqref{eq:4.7} and \eqref{eq:4.10} lead us to
$$\int_{-\infty}^{\infty}\left|\mathbb{L}_k(x,y) \right|\,\textrm{d}x\lesssim2^{-\frac{n}{2(n+2)}k}.$$
And it is easy to check that
$$\int_{-\infty}^{\infty}\left|\mathbb{L}_k(x,y) \right|\,\textrm{d}y\lesssim2^{-\frac{n}{2(n+2)}k}.$$
By interpolation and \eqref{4.8}, we have
$$\sum_{k\geq 0}\|\tilde{\mathbb{S}}_{k b}\|_{L^2(\mathbb{R})\rightarrow L^2(\mathbb{R})}\lesssim \sum_{k\geq 0} 2^{-\frac{n}{4(n+2)}k}\lesssim 1.$$
This finishes the proof of Proposition \ref{pro:4.2}.
\end{proof}

\begin{proposition}\label{pro:4.1} For fixed $x<y$ and for any $r\in \mathbb{R}$ satisfies $1\leq r-y < r-x \leq 2$, and $z \notin  E_k$, $J^r$ is defined in \eqref{eq:4.4}. We have the following estimates.
\begin{enumerate}
\item If $\left | \frac{\tilde{P}_k(z)}{\tilde{P}_k'(z)}\right |> \frac{4}{C_1}$, then
\begin{align}\label{eq:4.2}
|J^r|\leq C \left(\frac{1}{2^{nk}|x-y|}\right)^{\frac{1}{n}}.
\end{align}

\item If $\left(\frac{\tilde{P}_k}{\tilde{P}_k'}\right)'(z)> \frac{1}{8n}$, then
\begin{align}\label{eq:4.3}
|J^r|\leq C \left(\frac{1}{2^{nk}|x-y|}\right)^{\frac{1}{n+1}}.
\end{align}
\end{enumerate}
Here the bound $C$ are independent of the coefficients of polynomial $P$ and dependent only on $n$ and $C_1$.
\end{proposition}

\begin{proof}We recall that for fixed $x,y\in \mathbb{R}$, and for any $z\in \mathbb{R}$ satisfies $1\leq z-y < z-x \leq 2$,
$$\phi(z)=2^{nk} \frac{\gamma (\omega 2^k)}{\gamma (\omega)}  \left(\frac{\gamma (\omega 2^k(z-x))}{\gamma (\omega 2^k)} -\frac{\gamma (\omega 2^k(z-y))}{\gamma (\omega 2^k)} \right),$$
$$\psi(z)=\phi(z)\tilde{P}_k(z).$$
It is easy to obtain
$$\phi'(z)=2^{nk} \frac{\gamma (\omega 2^k)}{\gamma (\omega)}  \left(\frac{\omega 2^k\gamma '(\omega 2^k(z-x))}{\gamma (\omega 2^k)} -\frac{\omega 2^k\gamma' (\omega 2^k(z-y))}{\gamma (\omega 2^k)} \right)$$
and
$$\psi'(z)=\phi(z)\tilde{P}_k'(z)+\phi'(z)\tilde{P}_k(z).$$

It is nature that the critical points of the phase function in $J^r$ will be the obstacle to obtain the estimates \eqref{eq:4.2} and \eqref{eq:4.3}. For this reason, we introduce the following set $\Delta$. In this set there is no critical points of the phase function and its extra set is not too large.
Let $U$ be the union of the set of roots of $\tilde{P}_k'$ and of $\tilde{P}_k''$ over $\mathbb{R}$. For fixed $x,y\in \mathbb{R}$, for any $r\in \mathbb{R}$ satisfies $1\leq r-y < r-x \leq 2$ and any $\delta >0$, let
$$\Delta:=\{z\in(y+1,r):\  z \notin  E_k \,\textrm{and}\, \textrm{dist}(z, U)>\delta   \} . $$
Then
$$|(y+1, r)\setminus E_k\setminus \Delta|\leq C_{(n)}\delta$$
and
\begin{align}\label{eq:4.15}
|J^r|
   =& \left| \int_{(y+1,r)\backslash E_k} e^{i\psi(z)}\,\textrm{d}z\right|\\
 =& \left | \int_{\Delta} e^{i\psi(z)}\,\textrm{d}z+ \int_{(y+1,r)\backslash E_k \backslash \Delta} e^{i\psi(z)}\,\textrm{d}z\right|\nonumber\\
  \leq& \left | \int_{\Delta} e^{i\psi(z)}\,\textrm{d}z\right |+ \left|\int_{(y+1,r)\backslash E_k \backslash \Delta} e^{i\psi(z)}\,\textrm{d}z\right|.\nonumber
\end{align}
The last integral then is controlled by $C_{(n)}\delta$. It suffices to estimate
$$\left | \int_{\Delta} e^{i\psi(z)}\,\textrm{d}z\right |.$$

{\bf Case 1} $\left | \frac{\tilde{P}_k(z)}{\tilde{P}_k'(z)}\right |> \frac{4}{C_1}$ .
For $z\in \Delta$, we can write
\begin{align}\label{eq:4.11}
\frac{\psi'(z)}{\tilde{P}_k'(z)\phi'(z)}=\frac{\gamma(\omega 2^k (z-x))-\gamma(\omega 2^k(z-y))}{\omega 2^k\gamma'(\omega 2^k(z-x))-\omega 2^k\gamma'(\omega 2^k(z-y))}+\frac{\tilde{P}_k(z)}{\tilde{P}_k'(z)}.
\end{align}
By Lemma \ref{le2.2}, we get
\begin{eqnarray*}
 \left| \frac{\psi'(z)}{\tilde{P}_k'(z)\phi'(z)}\right|
   &\geq& \left|\frac{\tilde{P}_k(z)}{\tilde{P}_k'(z)}\right|- \left| \frac{\gamma(\omega 2^k (z-x))-\gamma(\omega 2^k(z-y))}{\omega 2^k\gamma'(\omega 2^k(z-x))-\omega 2^k\gamma'(\omega 2^k(z-y))}\right|\\
    &\geq&  \frac{4}{C_1}-\frac{2}{C_1}
     = \frac{2}{C_1}.
\end{eqnarray*}

On the other hand, by the generalised mean value theorem, $\gamma$ and $\gamma'$ are increasing on $(0,\infty)$, which yields $\frac{t\gamma'(t)}{\gamma(t)}\geq 1$ for any $t\in (0,\infty)$, and
\begin{eqnarray*}
 |\phi'(z)|
   &=& \left|2^{nk} \frac{\gamma (\omega 2^k)}{\gamma (\omega)}  \left(\frac{\omega 2^k\gamma '(\omega 2^k(z-x))}{\gamma (\omega 2^k)} -\frac{\omega 2^k\gamma' (\omega 2^k(z-y))}{\gamma (\omega 2^k)} \right)\right|\\
     &=& 2^{nk} \frac{\gamma (\omega 2^k)}{\gamma (\omega)} \frac{\omega 2^k}{\gamma (\omega 2^k)} \gamma ''(\omega 2^k\theta_2)\omega 2^k|x-y|\\
     &=& 2^{nk} \frac{\gamma (\omega 2^k)}{\gamma (\omega)} \frac{\omega 2^k\gamma' (\omega 2^k)}{\gamma (\omega 2^k)} \frac{\omega 2^k\theta_2\gamma ''(\omega 2^k\theta_2)}{\gamma' (\omega 2^k\theta_2)}\frac{1}{\theta_2}\frac{\gamma' (\omega 2^k\theta_2) }{\gamma' (\omega 2^k)}|x-y|\\
     &\gs& 2^{nk}|x-y|.
\end{eqnarray*}
where $1\leq z-y\leq \theta_2 \leq z-x\leq2$.
And Lemma \ref{le2.3} leads to
\begin{align}\label{eq:4.12}
|\psi'(z)|\gs 2^{nk}|x-y| \delta^{n-1}.
\end{align}
Then
\begin{eqnarray*}
\int_{\Delta}e^{i\psi(z)}\,\textrm{d}z
   &=& \int_{\Delta}   \frac{\textrm{d}}{\textrm{d}z}  \left( \frac{e^{i\psi(z)}}{\tilde{P}_k'(z)\phi'(z)}\right)  \frac{\tilde{P}_k'(z)\phi'(z)}{i\psi'(z)}\,\textrm{d}z +\int_{\Delta}  \frac{ e^{i\psi(z)}(\tilde{P}_k'(z)\phi'(z))' }{(\tilde{P}_k'(z)\phi'(z))^2}  \frac{\tilde{P}_k'(z)\phi'(z)}{i\psi'(z)}\,\textrm{d}z \\
    &=:& J_1+J_2.
\end{eqnarray*}
And
\begin{eqnarray*}
J_1
   &=&  \int_{\Delta}   \frac{\textrm{d}}{\textrm{d}z}  \left( \frac{e^{i\psi(z)}}{\tilde{P}_k'(z)\phi'(z)}\right)   \frac{\tilde{P}_k'(z)\phi'(z)}{i\psi'(z)}\,\textrm{d}z\\
    &=&  \left[ \frac{e^{i\psi(z)}}{i\psi'(z)} \right]_{\partial \Delta}- \int_{\Delta}  \frac{e^{i\psi(z)}}{\tilde{P}_k'(z)\phi'(z)}    \frac{\textrm{d}}{\textrm{d}z}  \left( \ \frac{\tilde{P}_k'(z)\phi'(z)}{i\psi'(z)}\right)  \, \textrm{d}z\\
    &=:& J_{1a}+J_{1b}.
\end{eqnarray*}
From \eqref{eq:4.12}, it follows that
$$|J_{1a}|\lesssim \frac{1}{|x-y|} \delta ^{1-n}2^{-nk}.$$
We have $ \left | \frac{\tilde{P}_k'(z)}{\tilde{P}_k(z)}\right |< \frac{C_1}{4}$ , by Lemma 2.2,
\begin{eqnarray*}
\left|\frac{\psi'(z)}{\tilde{P}_k(z)\phi'(z)}\right|
    &=& \left| \frac{\gamma(\omega 2^k (z-x))-\gamma(\omega 2^k(z-y))}{\omega 2^k\gamma'(\omega 2^k(z-x))-\omega 2^k\gamma'(\omega 2^k(z-y))}\frac{\tilde{P}_k'(z)}{\tilde{P}_k(z)}+1 \right|\\
    &\geq& 1- \left| \frac{\gamma(\omega 2^k (z-x))-\gamma(\omega 2^k(z-y))}{\omega 2^k\gamma'(\omega 2^k(z-x))-\omega 2^k\gamma'(\omega 2^k(z-y))}\frac{\tilde{P}_k'(z)}{\tilde{P}_k(z)} \right|\\
     &\geq& 1- \frac{2}{C_1}\frac{C_1}{4}
     = \frac{1}{2}.
\end{eqnarray*}
From $ \frac{\textrm{d}}{\textrm{d}z} \frac{\tilde{P}_k(z)}{\tilde{P}_k'(z)}=\frac{\tilde{P}_k'(z)\tilde{P}_k'(z)-\tilde{P}_k(z)\tilde{P}_k''(z)}{\tilde{P}_k'(z)\tilde{P}_k'(z)}$, it implies $\frac{\textrm{d}}{\textrm{d}z} \frac{\tilde{P}_k(z)}{\tilde{P}_k'(z)}$ has at most $2(n-1)$ roots. Combine Lemma \ref{le2.5},
\begin{eqnarray*}
|J_{1b}|
   &\leq&  \int_{\Delta}  \left|\frac{e^{i\psi(z)}}{\tilde{P}_k'(z)\phi'(z)}    \frac{\textrm{d}}{\textrm{d}z}  \left( \ \frac{\tilde{P}_k'(z)\phi'(z)}{i\psi'(z)}\right) \right| \, \textrm{d}z\\
    &\lesssim&   \frac{1}{|x-y|} 2^{-nk}\delta ^{1-n}\int_{\Delta} \left| \frac{\textrm{d}}{\textrm{d}z}  \left( \ \frac{\tilde{P}_k'(z)\phi'(z)}{\psi'(z)}\right) \right| \, \textrm{d}z\\
    &\approx&  \frac{1}{|x-y|} 2^{-nk}\delta ^{1-n}\int_{\Delta} \left|  \left(\frac{\tilde{P}_k'(z)\phi'(z)}{\psi'(z)} \right)^2   \frac{\textrm{d}}{\textrm{d}z}  \left( \Upsilon(z)+\frac{\tilde{P}_k(z)}{\tilde{P}_k'(z)}\right) \right| \, \textrm{d}z\\
    &\lesssim&  \frac{1}{|x-y|} 2^{-nk}\delta ^{1-n}  \left(\int_{\Delta}  \left|    \frac{\textrm{d}}{\textrm{d}z} \Upsilon(z)\right| \textrm{d}z + \int_{\Delta} \left(\frac{\tilde{P}_k'(z)\phi'(z)}{\psi'(z)} \right)^2 \left|  \frac{\textrm{d}}{\textrm{d}z} \frac{\tilde{P}_k(z)}{\tilde{P}_k'(z)} \right|  \,\textrm{d}z \right)\\
    &\lesssim&  \frac{1}{|x-y|}2^{-nk} \delta ^{1-n}  \left(1 + \int_{\Delta}\left(\frac{\tilde{P}_k(z)\phi'(z)}{\psi'(z)} \right)^2 \left|  \frac{\textrm{d}}{\textrm{d}z} \frac{\tilde{P}_k'(z)}{\tilde{P}_k(z)} \right|  \,\textrm{d}z \right)\\
    &\lesssim&  \frac{1}{|x-y|}2^{-nk} \delta ^{1-n}  \left(1 + \int_{\Delta} \left|  \frac{\textrm{d}}{\textrm{d}z} \frac{\tilde{P}_k'(z)}{\tilde{P}_k(z)} \right| \, \textrm{d}z \right)\\
    &\lesssim&  \frac{1}{|x-y|}2^{-nk} \delta ^{1-n}.
\end{eqnarray*}
From $ \frac{\textrm{d}}{\textrm{d}z} \frac{1}{\tilde{P}_k'(z)}=-\frac{\tilde{P}_k''(z)}{\tilde{P}_k'(z)\tilde{P}_k'(z)}$, this implies $\frac{\textrm{d}}{\textrm{d}z} \frac{1}{\tilde{P}_k'(z)}$ has at most $n-2$ roots. From $ \frac{\textrm{d}}{\textrm{d}z} \frac{1}{\phi'(z)}=-\frac{\phi''(z)}{\phi'(z)\phi'(z)}$ and $\phi''(z)=2^{nk} \frac{(\omega 2^k)^2}{\gamma (\omega)}  \left(\gamma ''(\omega 2^k(z-x)) -\gamma ''(\omega 2^k(z-y)) \right)$, noticing $\gamma''$ is monotone on $(0,\infty)$, we may get that the sign of $\frac{\textrm{d}}{\textrm{d}z} \frac{1}{\phi'(z)}$ does not change on this domain. From Lemma 2.5 and together with $|\frac{1}{\tilde{P}_k'(z)}|\lesssim\delta ^{1-n} $ and $|\frac{1}{\phi'(z)}|\lesssim\frac{1}{|x-y|}2^{-nk} $ enable us to obtain
\begin{eqnarray*}
|J_2|
   &\leq& \int_{\Delta} \left|  \frac{ e^{i\psi(z)}(\tilde{P}_k'(z)\phi'(z))' }{(\tilde{P}_k'(z)\phi'(z))^2}  \frac{\tilde{P}_k'(z)\phi'(z)}{i\psi'(z)} \right|   \,\textrm{d}z\\
    &\lesssim&   \int_{\Delta} \left|  \frac{ (\tilde{P}_k'(z)\phi'(z))' }{(\tilde{P}_k'(z)\phi'(z))^2}  \right| \,  \textrm{d}z\\
    &\lesssim&   \int_{\Delta} \left|  \frac{ \tilde{P}_k''(z) }{(\tilde{P}_k'(z))^2\phi'(z)}  \right|   \textrm{d}z+\int_{\Delta} \left|  \frac{ \phi''(z) }{\tilde{P}_k'(z)(\phi'(z))^2}  \right|   \,\textrm{d}z\\
     &\lesssim&  \frac{1}{|x-y|} 2^{-nk}  \int_{\Delta} \left|  \frac{\textrm{d}}{\textrm{d}z} \left(\frac{1}{\tilde{P}_k'(z)}\right) \right|   \textrm{d}z+ \delta ^{1-n}\int_{\Delta} \left|  \frac{\textrm{d}}{\textrm{d}z} \left(\frac{1}{\phi'(z)}\right)  \right|  \, \textrm{d}z\\
    &\lesssim&  \frac{1}{|x-y|} \delta ^{1-n}2^{-nk} .
\end{eqnarray*}

Therefore,

\begin{align}\label{eq:4.14}
\left|\int_{\Delta}e^{i\psi(z)}\,\textrm{d}z\right|\leq |J_{1a}|+|J_{1b}|+|J_2|\lesssim \frac{1}{|x-y|} \delta ^{1-n}2^{-nk} .
\end{align}

Let $\frac{1}{|x-y|} \delta ^{1-n}2^{-nk}\approx  C_{(n)}\delta$, then we obtain \eqref{eq:4.2}.

{\bf Case 2} $\left(\frac{\tilde{P}_k}{\tilde{P}_k'}\right)'(z)> \frac{1}{8n} $.
 We denote $\Delta:=\bigcup_{m=1}^{C_{(n)}}(a_m,b_m)$. On each $(a_m,b_m)$, by Lemma \ref{le2.2}, the derivative of $\frac{\psi'(z)}{\tilde{P}_k'(z)\phi'(z)}$ is greater than $\frac{1}{8n}$. Then $\frac{\psi'(z)}{\tilde{P}_k'(z)\phi'(z)}$ is strictly increasing on this domain.
 If it has one (and only one) zero in this interval, we denote it as $z_m\in (a_m, b_m)$. By mean value theorem,
\begin{equation}\label{eq:4.20}\left|\frac{\psi'(z)}{\tilde{P}_k'(z)\phi'(z)}\right|\geq \frac{1}{8n} |z-z_m|.
\end{equation}
  If $\frac{\psi'(z)}{\tilde{P}_k'(z)\phi'(z)}$ has no zero in $\left(a_m,b_m\right)$, we consider two cases: If $\frac{\psi'(z)}{\tilde{P}_k'(z)\phi'(z)}>0$ on $\left(a_m,b_m\right)$, we take $z_m$ as the intersection between $z$-axis and the tangent line to the function $\frac{\psi'(z)}{\tilde{P}_k'(z)\phi'(z)}$ at $z=a_m$; If $\frac{\psi'(z)}{\tilde{P}_k'(z)\phi'(z)}<0$ on $\left(a_m,b_m\right)$, we take $z_m$  as the intersection between $z$-axis and the tangent line to the function $\frac{\psi'(z)}{\tilde{P}_k'(z)\phi'(z)}$ at $z=b_m$. Then, we obtain a series $\{z_m\}_{m=1}^{C_{(n)}}$ such that, for each $m\in[1,C_{(n)}]$ and $m\in \mathbb{Z}$, we have \eqref{eq:4.20} is true.

Let $B_\delta:=\left\{z\in \Delta:\ \textrm{dist}\left(z, \bigcup_{m=1}^{C_{(n)}} \left\{z_m\right\}\right)\leq\delta\right\}$ and $D:=\Delta \setminus B_\delta$. We know that $|B_\delta|\leq C_{(n)}\delta$ and $D$ consists of $C_{(n)}$ intervals. Then we focus our goal to estimate
$$\left|\int_{D}e^{i\psi(z)}\,\textrm{d}z\right|.$$

For $z\in D$, as in Case 1,
\begin{align}\label{eq:4.13}
|\psi'(z)|\gs 2^{nk}|x-y| \delta^n.
\end{align}
Replacing the $\Delta$ in \eqref{eq:4.14} by $D$ and running the same argument, we can see that $J_{1a}\lesssim \frac{1}{|x-y|}2^{-nk} \delta ^{-n}$ and $J_2\lesssim \frac{1}{|x-y|}2^{-nk}\delta ^{-n}$. For $J_{1b}$, noticing $\frac{\psi'(z)}{\tilde{P}_k'(z)\phi'(z)}$ is increasing on $D$, we have
\begin{eqnarray*}
|J_{1b}|
   &\leq&  \int_{D}  \left|\frac{e^{i\psi(z)}}{\tilde{P}_k'(z)\phi'(z)}    \frac{\textrm{d}}{\textrm{d}z}  \left( \ \frac{\tilde{P}_k'(z)\phi'(z)}{i\psi'(z)}\right) \right| \, \textrm{d}z\\
    &\lesssim&   \frac{1}{|x-y|}2^{-nk} \delta ^{1-n}\int_{D} \left| \frac{\textrm{d}}{\textrm{d}z}  \left( \ \frac{\tilde{P}_k'(z)\phi'(z)}{\psi'(z)}\right) \right|  \,\textrm{d}z\\
    &\approx&  \frac{1}{|x-y|} 2^{-nk}\delta ^{1-n}\left|\int_{D}  \frac{\textrm{d}}{\textrm{d}z}  \left( \ \frac{\tilde{P}_k'(z)\phi'(z)}{\psi'(z)}\right)  \textrm{d}z\right| \\
    &\lesssim&  \frac{1}{|x-y|}2^{-nk} \delta ^{-n}.
\end{eqnarray*}
Therefore,
\begin{align}\label{eq:4.16}
\left|\int_{D}e^{i\psi(z)}\,\textrm{d}z\right|\leq |J_{1a}|+|J_{1b}|+|J_2|\lesssim \frac{1}{|x-y|} \delta ^{-n}2^{-nk} .
\end{align}
As in \eqref{eq:4.15}, $|J^r|$ can be controlled by
\begin{align}\label{eq:4.17}
|J^r|\lesssim \left(\frac{1}{2^{nk}|x-y|}\right)^{\frac{1}{n+1}}.
\end{align}
This is \eqref{eq:4.3}, we finish the proof of Propostion \ref{pro:4.1}.

\end{proof}

\bigskip

\noindent  Junfeng Li and Haixia Yu (Corresponding author)

\smallskip

\noindent  Laboratory of Mathematics and Complex Systems
(Ministry of Education of China),
School of Mathematical Sciences, Beijing Normal University,
Beijing 100875, People's Republic of China

\smallskip

\noindent {\it E-mails}: \texttt{lijunfeng@bnu.edu.cn} (J. Li)

\noindent\phantom{{\it E-mails:}} \texttt{yuhaixia@mail.bnu.edu.cn} (H. Yu)

\bigskip

\end{document}